\documentclass[12pt,letterpaper]{article}
\usepackage{amsfonts}
\usepackage{amssymb}
\usepackage{latexsym}  

\let\SavedRightarrow=\Rightarrow
\usepackage{marvosym}
\let\Rightarrow=\SavedRightarrow

\usepackage{amsmath}   

\newcommand\LL{{\mathcal L}}
\newcommand\PP{{\mathcal P}}
\newcommand\II{{\mathcal I}}
\newcommand\EE{{\mathcal E}}
\newcommand\DD{{\mathcal D}}
\newcommand\WW{{\mathcal W}}

\newcommand\QQQ{{\mathbb Q}}
\newcommand\RRR{{\mathbb R}}

\newcommand\cccc{{\mathfrak c}}

\newcommand\cchi{{\raise 2 pt \hbox{$\chi$}}}

\newcommand{\MS}{\mathsf{MS}}  

\newcommand\cl{\mathrm{cl}}   
\newcommand\Lim{\mathrm{Lim}}  

\newcommand\tleft{\vartriangleleft}  
\newcommand\cat{^{\mathord{\frown}}}  

\newcommand\rest{\upharpoonright}     
\newcommand\res{\mathord {\upharpoonright}}  

\newcommand\iv{^{-1}} 

\newcommand\onto{\twoheadrightarrow}

\newcommand\eop{{\Large \Coffeecup}}  

{\end{enumerate}}
\newenvironment{itemizz}{\begin{itemize}\setlength{\itemsep}{-1mm}} %
{\end{itemize}}                              

\newcounter{constr}
{\end{itemize}}

\newcounter{ctr}
\newenvironment{enumctr}{\begin{itemize}\setlength{\itemsep}{-1mm}} %
{\end{itemize}}
\newcommand\itemctr{\addtocounter{ctr}{1}\item[\arabic{ctr}.]}

\newtheorem{theorem}{Theorem}[section]
\newtheorem{definition}[theorem]{Definition}
\newtheorem{notation}[theorem]{Notation}
\newtheorem{lemma}[theorem]{Lemma}

\newtheorem{question}[theorem]{Question}
\newtheorem{conditions}[theorem]{Conditions}
\newtheorem{corollary}[theorem]{Corollary}
\newtheorem{proposition}[theorem]{Proposition}

\newenvironment{proof}{{\bf Proof.}}{\eop\medskip}
\newenvironment{proofof}[1]{\medskip \textbf{Proof of #1.}}{\eop\medskip}

\begin{document}

\title{Aronszajn Compacta
}

\author{Joan E. Hart\footnote{University of Wisconsin, Oshkosh,
WI 54901, U.S.A.,
\ \ hartj@uwosh.edu}
\  and
Kenneth Kunen\footnote{University of Wisconsin,  Madison, WI  53706, U.S.A.,
\ \ kunen@math.wisc.edu}
\thanks{Both authors partially supported by NSF Grant
DMS-0456653.}
}

\maketitle

\begin{abstract}
We consider a class of compacta $X$ such that the maps from $X$
onto metric compacta define an
Aronszajn tree of closed subsets of $X$.
\end{abstract}

\section{Introduction} 
\label{sec-intro}
All topologies discussed in this paper are assumed to be Hausdorff.
We begin by defining an \emph{Aronszajn compactum},
along with a natural tree structure,
by considering a space embedded into a cube.
An equivalent definition, in terms of elementary submodels,
is considered in Section \ref{sec-elem}.

\begin{notation}
\label{not-basic}
Given a product $\prod_{\xi < \lambda} K_\xi$:
If $\alpha \le \beta \le \lambda$, then 
$\pi^\beta_\alpha$ denotes the natural projection from
$\prod_{\xi < \beta} K_\xi$ onto
$\prod_{\xi < \alpha} K_\xi$.
If we are studying a space $X \subseteq \prod_{\xi < \lambda} K_\xi$
then $X_\alpha$ denotes $\pi^\lambda_\alpha(X)$,
and $\sigma^\beta_\alpha$ denotes the restricted map
$\pi^\beta_\alpha \res X_\beta$; so
$\sigma^\beta_\alpha: X_\beta \onto X_\alpha$.
\end{notation}

\begin{definition}
\label{def-emb-ar}
An \emph{embedded Aronszajn compactum} is a closed subspace
$X \subseteq [0,1]^{\omega_1}$ with
$w(X) = \aleph_1$ and $\cchi(X) = \aleph_0$ such that for some club
$C \subseteq \omega_1$:
for each $\alpha \in C$
$\LL_\alpha := \{x \in X_\alpha : |(\sigma^{\omega_1}_\alpha)\iv\{x\}| > 1\}$
is countable.  For each such $X$, define
$T=T(X) := \bigcup\{\LL_\alpha : \alpha \in C\}$, and
let $\tleft$ denote the following order:
if $\alpha, \beta \in C$, $\alpha < \beta$, $x \in \LL_\alpha$
and $y \in \LL_\beta$, then $x \tleft y$ iff $x = \pi^\beta_\alpha(y)$.
\end{definition}

The $\sigma^{\omega_1}_\alpha$ for which 
$|\LL_\alpha| \le \aleph_0$ are called
\textit{countable rank} maps in \cite{Fed,KT}.
Observe that $\langle T(X),\tleft \rangle $
is a tree.
Each level $\LL_\alpha$ is countable by definition, and is non-empty
because $w(X) = \aleph_1$; then $T$ is Aronszajn because $\cchi(X) = \aleph_0$.
Of course, a compactum of weight $\aleph_1$ may be embedded into
$[0,1]^{\omega_1}$ in many ways, but:

\begin{lemma}
\label{lemma-embed}
If $X,Y \subseteq [0,1]^{\omega_1}$, $X$ is
an embedded Aronszajn compactum, and $Y$ is homeomorphic to $X$,
then $Y$ is an embedded Aronszajn compactum.
\end{lemma}
\begin{proof}
Let $f: X \to Y$ be a homeomorphism.  Then use the fact that there
is a club $D \subseteq \omega_1$ on which $f$ commutes with projection;
that is, for $\alpha \in D$,
there is a homeomorphism $f_\alpha : X_\alpha \onto Y_\alpha$ such that
$\pi^{\omega_1}_\alpha \circ f  = f_\alpha \circ \pi^{\omega_1}_\alpha$.
\end{proof}

The proof of this lemma shows that the Aronszajn
trees derived from $X$ and from $Y$ are isomorphic on a club.

\begin{definition}
An \emph{Aronszajn compactum} is a
compact $X$ such that
$w(X) = \aleph_1$ and $\cchi(X) = \aleph_0$ and for some
\textup(equivalently, for all\,\textup) $Z \subseteq [0,1]^{\omega_1}$
homeomorphic to $X$, $Z$ is an embedded Aronszajn compactum.
\end{definition}

The next lemma is immediate from the definition.
Further closure properties of the class of Aronszajn compacta
are considered in Section \ref{sec-clos}.

\begin{lemma}
\label{lemma-subset}
A closed subset of an Aronszajn compactum is
either second countable or an Aronszajn compactum.
\end{lemma}

The Dedekind completion of an Aronszajn line is an Aronszajn compactum
(see Section \ref{sec-elem}),
and the associated tree is essentially
the same as the standard tree of closed intervals.
A special case of this is a compact Suslin line, which is
a well-known compact L-space; that is, it is 
HL (hereditarily Lindel\"of) and not
HS (hereditarily separable).  The line derived from a special
Aronszajn tree is much different topologically,
since it is not even ccc.

In Section \ref{sec-constr} we shall prove:

\begin{theorem}
\label{thm-Aronsz-basic}
Assuming $\diamondsuit$, there is an Aronszajn compactum  
which is both HS and HL.
\end{theorem}

Our construction is flexible enough 
to build in additional properties for the space and its associated tree,
which may be either Suslin or special; see Theorem \ref{thm-Aronsz-refined}.
The form of the tree is (up to club-isomorphism)
a topological invariant of $X$, but seems to be unrelated
to more conventional topological properties of $X$;
for example, $X$ may be totally
disconnected, or it may be connected and locally connected,
with $\dim(X)$ finite or infinite.

\begin{question}
Is there, in ZFC, an HL Aronszajn compactum?
\end{question}

We would expect a ZFC example to be both HS and HL.
Note that an Aronszajn compactum is dissipated in the sense of \cite{KU2},
so it cannot be an L-space
if there are no Suslin lines by Corollary 5.3 of \cite{KU2}.

To refute the existence of an HL Aronszajn compactum,
one needs more than just an Aronszajn tree of closed sets, 
since this much exists in the Cantor set:

\begin{proposition}
There is an Aronszajn tree $T$ whose nodes are closed subsets
of the Cantor set $2^\omega$.  The tree ordering is $\supset$,
with root $2^\omega$. Each level of $T$
consists of a pairwise disjoint family of sets.
\end{proposition}

The proof is like that of Theorem 4 of Galvin and Miller \cite{GM},
which is attributed there to Todor\v cevi\'c.

\section{Elementary Submodels}
\label{sec-elem}

We consider Aronszajn compacta from the point of view
of elementary submodels.
Assume that $X$ is compact, with $X$ (and its topology)
in some suitably large $H(\theta)$.
If $X$ is first countable, so that $|X| \le \cccc$ and its
topology is a set of size $\le 2^\cccc$, then $\theta$ can
be any regular cardinal larger than $2^\cccc$,
assuming that the set $X$ is chosen so that its transitive
closure has size $\le \cccc$.

If $X \in M \prec H(\theta)$, 
then there is a natural quotient map $\pi = \pi_M: X \onto X/M$
obtained by identifying two points of $X$ iff they are not
separated by any function in $C(X,\RRR) \cap M$.
Furthermore, $X/M$ is second countable whenever $M$ is countable.

\begin{lemma}
\label{lemma-el}
Assume that $X$ is compact,
$w(X) = \aleph_1$, and $\cchi(X) = \aleph_0$.
Then the following are equivalent:
\begin{itemizz}
\item[1.] $X$ is an Aronszajn compactum.
\item[2.] Whenever $M$ is countable and
$X \in M \prec H(\theta)$,
there are only countably many $y \in X/M$
such that $\pi\iv\{y\}$ is not a singleton.
\item[3.] $(2)$ holds for all $M$ in some club of
countable elementary submodels of $H(\theta)$.
\end{itemizz}
\end{lemma}
\begin{proof}
For $(1)\to(2)$, note that $X \in M \prec H(\theta)$ implies
that $M$ contains some club satisfying Definition \ref{def-emb-ar}.
\end{proof}

For example, say that $X$ is a compact first countable LOTS.
Then the equivalence classes are all convex; and,
if $x < y$ then $\pi(x) = \pi(y)$ iff $[x,y] \cap M = \emptyset$.
Now consider Aronszajn lines:

\begin{definition}
\label{def-ar-line}
A \emph{compacted Aronszajn line} is a compact LOTS $X$
such that $w(X) = \aleph_1$ and $\cchi(X) = \aleph_0$
and the closure of every countable
set is second countable.
\end{definition}

By $\cchi(X) = \aleph_0$,
there are no increasing or decreasing $\omega_1$--sequences.
Note that our definition allows for the possibility
that $X$ contains uncountably many disjoint
intervals isomorphic to $[0,1]$.
The term ``\emph{compact} Aronszajn line'' is not
common in the literature.  An \emph{Aronszajn line} is usually
defined to be a LOTS of size $\aleph_1$
with no increasing or decreasing $\omega_1$--sequences
and no uncountable subsets of \emph{real type}
(that is, order-isomorphic to a subset of $\RRR$).
Such a LOTS cannot be compact; the Dedekind completions
of such LOTSes are the compact\emph{ed} Aronszajn lines of
Definition \ref{def-ar-line}.

\begin{lemma}
\label{lemma-ar-line}
A LOTS $X$ is an Aronszajn compactum iff $X$ is a compacted Aronszajn line.
\end{lemma}
\begin{proof}
For $\leftarrow$: suppose that
$X \in M \prec H(\theta)$ and $M$ is countable.
Then $X/M$ is a compact metric LOTS, and is hence order-embeddable
into $[0,1]$.  Suppose there were an uncountable $E \subseteq X/M$
such that $|\pi\iv\{y\}| \ge 2$ for all $y \in E$.
Say $\pi\iv\{y\} = [a_y, b_y] \subset X$ for $y \in E$,
where $a_y < b_y$.  If $D$ is a countable dense subset of $E$
then $\cl(\{a_y : y \in D\}) \subseteq X$ would not be second
countable, a contradiction.
\end{proof}

We use the standard definition of a \emph{Suslin line} as any
LOTS which is ccc and not separable; this is always an L-space.
Then a \emph{compact Suslin line} is just a Suslin line which
happens to be compact.  A compacted Aronszajn line may be a Suslin line,
but a compact Suslin line need not be a
compacted Aronszajn line.  For example, we may form $X$
from a connected compact Suslin line $Y$
by doubling uncountably many points lying in some Cantor subset
of $Y$.  More generally, 

\begin{lemma}
\label{lemma-sus-ar}
Let $X$ be a compact Suslin line.  Then $X$ is 
a compacted Aronszajn line iff
$D := \{x \in X : \exists y > x \,( [x,y] = \{x,y\} \}$
does not contain an uncountable subset of real type.
\end{lemma}
\begin{proof}
Note that $D$ is the set of all points
with a right nearest neighbor.
If $D$ contains an uncountable set $E$ real type, 
let $B \subseteq E$ be countable and dense in $E$.
Then whenever $M$ is countable and
$X,B \in M \prec H(\theta)$,
there are uncountably many $y \in X/M$
such that $|\pi\iv\{y\}| \ge 2$,
so that $X$ is not an Aronszajn compactum.

Conversely, if $X$ is not an Aronszajn compactum,
consider any countable $M$ with
$X \in M \prec H(\theta)$ and
$A := \{y \in X/M : |\pi\iv\{y\}| \ge 2\}$ uncountable.
Let $A' := \{y \in X/M : |\pi\iv\{y\}| > 2\}$.
Since each $\pi\iv\{y\}$ is convex,
$A'$ is countable by the ccc, and the left points
of the $\pi\iv\{y\}$ for $y \in A \backslash A'$ yield an uncountable
subset of $D$ of real type.
\end{proof}

A zero dimensional compact Suslin line formed
in the usual way from a binary Suslin tree will
also be a compacted Aronszajn line.

\section{Normalizing Aronszajn Compacta}
\label{sec-norm}
The club $C$ and tree $T$ derived from an Aronszajn
compactum $X$ in Definition \ref{def-emb-ar} can depend on the
embedding of $X$ into $[0,1]^{\omega_1}$.
To standardize the tree, we  choose a nice embedding.  
For $X \subseteq [0,1]^{\omega_1}$, $C$ cannot in general be $\omega_1$,
since $C = \omega_1$ implies that $\dim(X) \le 1$.
Replacing  $[0,1]$ by the Hilbert cube, however, 
we can assume $C = \omega_1$, which simplifies our tree notation.
In particular, the levels will be indexed by $\omega_1$, 
so that $\LL_\alpha$ will be level $\alpha$ of the tree
in the usual sense.

\begin{definition}
\label{def-cube}
$Q$ denotes the Hilbert cube, $[0,1]^\omega$. If
$X \subseteq Q^{\omega_1}$ is closed and $\alpha< \omega_1$, then
$\LL_\alpha = \LL_\alpha(X)=
\{x \in X_\alpha : |(\sigma^{\omega_1}_\alpha)\iv\{x\}| > 1\}$.
$\WW(X) = \{\alpha< \omega_1: |\LL_\alpha| \le \aleph_0\}$.
\end{definition}

So, $X$ is an Aronszajn compactum iff $\WW(X)$ contains a club;
$\WW(X)$ itself need not be closed, and $\WW(X)$ depends
on how $X$ is embedded into $ Q^{\omega_1}$.
Now, using the facts that $Q \cong Q^\omega$ and that
an Aronszajn tree can have only countably many finite levels:

\begin{lemma}
\label{lemma-nf}
Every Aronszajn compactum is homeomorphic to some $X \subseteq Q^{\omega_1}$
such that $\WW(X) = \omega_1$ and
$|\LL_\alpha| = \aleph_0$ for all $\alpha > 0$.
\end{lemma}

Of course, $\LL_0 = X_0 = \{\emptyset\} = Q^0$,
and $\emptyset$ is the root node of the tree.

\begin{definition}
If $X \subseteq Q^{\omega_1}$ is an Aronszajn compactum
and $\WW(X) = \omega_1$, let $\widehat \LL _\alpha =
\{x \in \LL_\alpha : w((\sigma^{\omega_1}_\alpha)\iv\{x\}) = \aleph_1\}$,
and let $\widehat T = \bigcup_\alpha \widehat \LL _\alpha$.
\end{definition}

Since $X$ is not second countable, each
$\widehat \LL_\alpha  \ne \emptyset$ and $\widehat T$ is
an Aronszajn subtree of $T$.
Repeating the above argument, we get

\begin{lemma}
\label{lemma-better-nf}
Every Aronszajn compactum is homeomorphic to some $X \subseteq Q^{\omega_1}$
such that $\WW(X) = \omega_1$, and
$|\widehat \LL_\alpha| = \aleph_0$ for all $\alpha > 0$,
and each $x \in \LL_ \alpha \backslash \widehat \LL_\alpha$ is a
leaf, and each $x \in \widehat \LL_\alpha$ 
has $\aleph_0$ immediate successors in $\widehat \LL_{\alpha +1}$.
\end{lemma}

This normalization can also be obtained with elementary submodels.
Start with a continuous chain of elementary submodels,
$M_\alpha \prec H(\theta)$, for $\alpha < \omega_1$,
with $X \in M_0$ and each $M_\alpha \in M_{\alpha+1}$.
Let $X_\alpha = X / M_\alpha$, let $\pi_\alpha: X \onto X_\alpha$
be the natural map, and let
$\LL_\alpha = \{y \in X_\alpha: |\pi_\alpha\iv\{y\}| > 1\}$.
We may view each $X_\alpha$ as embedded topologically
into $Q^\alpha$, in which case $\LL_\alpha$ has the same meaning
as before.
If $\pi_\alpha\iv\{y\}$ is second countable, then
(since $M_\alpha \in M_{\alpha+1}$), all the points in 
$\pi_\alpha\iv\{y\}$ are separated by functions in
$C(X) \cap M_{\alpha+1}$, so 
$y \in \LL_ \alpha \backslash \widehat \LL_\alpha$ is a leaf.

If $X$ is a compacted Aronszajn line, then
$X_{\alpha + 1}$ is formed by replacing
each $y \in \LL_\alpha$ by a compact interval $I_y$ of size at 
least $2$.  If $y \in \LL_ \alpha \backslash \widehat \LL_\alpha$,
then $\pi_\alpha\iv\{y\}$ is second countable
and is isomorphic to $I_y$.
Note that the tree may have uncountably many leaves;
we do not obtain the conventional normalization of an Aronszajn tree,
where the tree is uncountable above every node.

Next, we consider the ideal of second countable subsets of $X$:

\begin{definition}
For any space $X$, $\II_X$ denotes the family of all $S \subseteq X$
such that $S$, with the subspace topology, is second countable.
\end{definition}

$\II_X$ need not be an ideal.  It is obviously closed under subsets,
but need not be closed under unions (consider $\omega \cup \{p\}
\subset \beta \omega $).

\begin{lemma}
\label{lemma-ideal}
Assume that $X \subseteq Q^{\omega_1}$ is an HL Aronszajn compactum,
as in Lemma \ref{lemma-nf}.  Then
$\II_X$  is a $\sigma$--ideal, and, for all $S \subseteq X$,
the following are equivalent:
\begin{itemizz}
\item[1.] $S \in \II_X$.
\item[2.] For some $\alpha < \omega_1$,
$\sigma^{\omega_1}_\alpha (S) \cap \LL_\alpha = \emptyset$.
\item[3.] There is a $G \supseteq S$ such that $G\in \II_X$
and $G$ is a $G_\delta$ subset of $X$.
\item[4.] There is an $f \in C(S,Q)$ such that $f$ is 1-1.
\item[5.] There is an $f \in C(S,Q)$ such that 
$f\iv\{y\}$ is second countable for all $y \in Q$.
\end{itemizz}
\end{lemma}
\begin{proof}
It is easy to verify $(2) \to (3) \to (1) \to (4) \to (5)$.
In particular, for $(2) \to (3)$:  Fix $\alpha$ and let
$G = (\sigma^{\omega_1}_\alpha)\iv (X_\alpha \backslash \LL_\alpha)$.
Then $G$ is a $G_\delta$ set and 
$\sigma^{\omega_1}_\alpha : G \onto X_\alpha \backslash \LL_\alpha$
is a 1-1 closed map, and hence a homeomorphism.

For $(1) \to (2)$:
Fix an open base for $S$ of the form 
$\{V_n \cap S : n \in \omega\}$, where each $V_n$ is open in $X$.
$X$ is HL, so $V_n$ is an $F_\sigma$.
We can thus fix $\xi < \omega_1$ such that each
$V_n = (\sigma^{\omega_1}_\xi)\iv ( \sigma^{\omega_1}_\xi (V_n))$.
It follows that $\sigma^{\omega_1}_\xi$ is 1-1 on $S$.
We may then choose $\alpha$ with $\xi < \alpha < \omega_1$
such that $\sigma^{\omega_1}_\alpha(S) \cap \LL_\alpha = \emptyset$.

Now, $\II_X$  is a $\sigma$--ideal by  $(1) \leftrightarrow (2)$.

To prove $(5) \to (2)$:  Fix $f$ as in (5).
Let $\{U_n : n \in \omega\}$ be an open base for $Q$;
then $f\iv(U_n) = S \cap V_n$, where $V_n$ is open in $X$
and hence an $F_\sigma$ set.
We can thus fix $\alpha < \omega_1$ such that
$V_n = (\sigma^{\omega_1}_\alpha)\iv ( \sigma^{\omega_1}_\alpha (V_n))$.
It follows that $f$ is constant on 
$S \cap (\sigma^{\omega_1}_\alpha)\iv \{z\}$ for all $z \in X_\alpha$.
Thus, 
$S \cap (\sigma^{\omega_1}_\alpha)\iv \{z\}$ is second
countable for all $z \in X_\alpha$.
But then $S$ is contained in the union
of $\bigcup\{S\cap ( \sigma^{\omega_1}_\alpha )\iv\{y\} : y \in \LL_\alpha\}$
and $ ( \sigma^{\omega_1}_\alpha )\iv (X_\alpha \backslash \LL_\alpha) \cong
(X_\alpha \backslash \LL_\alpha) $, so $S \in \II_X$ because
$\II_X$  is a $\sigma$--ideal.
\end{proof}

This proof shows that every Aronszajn compactum
is an ascending union of $\omega_1$ Polish spaces:
namely, the
$(\sigma^{\omega_1}_\alpha)\iv  (X_\alpha \backslash \LL_\alpha)$.

We needed $X$ to be Aronszajn in Lemma \ref{lemma-ideal};
HS and HL are not enough to prove the equivalence of (1)(3)(4)(5).
If $S$ is the
Sorgenfrey line contained in the double arrow 
space $X$, then (4)(5) are true but (1)(3) are false.
Similar remarks hold for similar spaces which are both HS and HL.
For example, assuming CH, Filippov \cite{Fil} constructed a
locally connected continuum which is HS and HL but not second
countable.  The space was obtained by replacing a Luzin set of
points in $[0,1]^2$ by circles.  If $S$ contains one point
from each of the circles, then S satisfies (4)(5) but
fails (1)(3).
In both examples, the space $X$ itself satisfies (5) but not (1)(3)(4).

More generally, any space $X$ that has an $f\in C(X,Q)$ as in (5) 
cannot be an Aronszajn compactum.  Thus, a ZFC example of an
HL Aronszajn compactum would
settle in the negative the following well-known question of Fremlin
(\cite{Fr} 44Qc): is it consistent that for every HL compactum $X$,
there is an $f \in C(X,Q)$ such that 
$|f\iv\{y\}| <  \aleph_0$  for all $y \in Q$?
In \cite{Gr}, Gruenhage gives some of the history 
related to this question, and points out
some related results suggesting that the answer might be
``yes'' under PFA.

\section{Closure Properties of Aronszajn Compacta}
\label{sec-clos}
Closure under subspaces was already mentioned in
Lemma \ref{lemma-subset}.
For products, Lemma
\ref{lemma-el} implies:

\begin{lemma}
\label{lemma-Ar-prod}
Assume that $X$ is an Aronszajn compactum and $Y$ is an arbitrary space.
Then $X \times Y$ is an Aronszajn compactum iff $Y$ is compact and countable.
\end{lemma}

Regarding quotients, we first prove:

\begin{lemma}
\label{lemma-eq-pair}
Assume that $X,Y$ are compact, $\varphi: X \onto Y$, and
$X,Y,\varphi \in M \prec H(\theta)$.   Let $\sim$ denote the $M$ equivalence relation
on $X$ and on $Y$.  Then 
\begin{itemizz}
\item[1.] If $x_0, x_1 \in X $ and $x_0 \sim x_1$, then 
$\varphi(x_0) \sim \varphi(x_1)$;
so, the inverse image of an equivalence class of $Y$
is a union of equivalence classes of $X$.
\item[2.] If $y_0, y_1 \in Y $ and $x_0 \not\sim x_1$ 
for all $x_0 \in \varphi\iv\{y_0\}$ and
all $x_1 \in \varphi\iv\{y_1\}$, then
$y_0 \not\sim y_1$.
\end{itemizz}
\end{lemma}
\begin{proof}
For (1):  If $f \in C(Y) \cap M$ separates $\varphi(x_0)$
from $ \varphi(x_1)$ then
$f \circ \varphi \in C(X) \cap M$ separates $x_0$ from $x_1$.

For (2):  For each 
$x_0 \in \varphi\iv\{y_0\}$ and $x_1 \in \varphi\iv\{y_1\}$,
there is an $f \in C(X, [0,1]) \cap M$ such that $f(x_0) \ne f(x_1)$.
By compactness of $\varphi\iv\{y_0\} \times \varphi\iv\{y_1\}$,
there are $f_0, \ldots, f_{n-1} \in C(X, [0,1]) \cap M$
for some $n\in \omega$ such that: for all 
$x_0 \in \varphi\iv\{y_0\}$ and $x_1 \in \varphi\iv\{y_1\}$,
there is some $i < n$ such that
$f_i(x_0) \ne f_i(x_1)$.
These yield an $\vec f \in C(X, [0,1]^n) \cap M$ such that
$\vec f ( \varphi\iv\{y_0\} )  \cap  \vec f ( \varphi\iv\{y_1\} ) =
\emptyset  $.
Since $M$ contains a base for $[0,1]^n$,
there are open $U_0, U_1 \subseteq [0,1]^n$ with each $U_i \in M$
such that $\overline{U_0} \cap \overline{U_1} = \emptyset$
and each $\vec f ( \varphi\iv\{y_i\} )  \subseteq U_i$, so that
$\varphi\iv\{y_i\} \subseteq (\vec{f}\,)\iv (U_i)$.
Let $V_i = \{y \in Y : \varphi\iv\{y\} \subseteq (\vec f)\iv (U_i) \}$.
Then the $V_i$ are open in $Y$, each $V_i \in M$, each $y_i \in V_i$, and 
$\overline{V_0} \cap \overline{V_1} = \emptyset$. 
There is thus a $g \in C(Y) \cap M$
such that $g(\overline{V_0}) \cap g(\overline{V_1}) = \emptyset$, 
so that $g(y_0) \ne g(y_1)$.  Thus, $y_0 \not\sim y_1$.
\end{proof}

\begin{theorem}
\label{thm-ar-quotient}
Assume that $X$ is an Aronszajn compactum,
$\varphi : X \onto Y$, 
$w(Y) = \aleph_1$, and $\cchi(Y) = \aleph_0$.
Then $Y$ is an Aronszajn compactum.
\end{theorem}
\begin{proof}
It is sufficient to check that for a club of elementary submodels $M$,
all but countably many $M$--classes of $Y$ are singletons.
Fix $M$ as in Lemma \ref{lemma-eq-pair}; so
all but countably many $M$--classes of $X$ are singletons.
Then for all but countably many classes $K = [y]$ of $Y$:
all $M$--classes of $X$ inside of $\varphi\iv(K)$ are singletons,
so that, by the lemma, $K$ is a singleton.
\end{proof}

Note that we needed to assume that $\cchi(Y) = \aleph_0$.
Otherwise, when $X$ is not HL, we would get a trivial counterexample
of the form $X/K$, where $K$ is a closed set which is not a $G_\delta$.

Examining whether an Aronszajn compactum may be both HS and HL reduces 
to considering zero dimensional spaces and connected spaces,
by the following lemma.

\begin{lemma}
Assume that $X$ is an HL Aronszajn compactum,
$\varphi : X \onto Y$.  Then either
$Y$ is an Aronszajn compactum or some $\varphi\iv\{y\}$ is an
Aronszajn compactum.
\end{lemma}
\begin{proof}
$Y$ will be an Aronszajn compactum unless it is second countable.
But if it is second countable, then some  $\varphi\iv\{y\}$
will be not second countable by Lemma \ref{lemma-ideal},
and then  $\varphi\iv\{y\}$ will be an Aronszajn compactum.
\end{proof}

\begin{corollary}
Suppose there is an Aronszajn compactum $X$ which is HS and HL.
Then there is an Aronszajn compactum $Z$ which is HS and HL
and which is either connected or zero dimensional.
\end{corollary}
\begin{proof}
Get $\varphi : X \onto Y$ by collapsing all connected components to points.
Then $Z$ is either $Y$ or some component.
\end{proof}

Note that the cone over $X$ is also connected, but is
not an Aronszajn compactum by Lemma \ref{lemma-Ar-prod}.

\section{Constructing Aronszajn Compacta}
\label{sec-constr} 

We begin this section by constructing a space $X$  which proves
Theorem \ref{thm-Aronsz-basic}.
We construct $X=X_{\omega_1}$ as an inverse limit as
a closed subspace of $Q^{\omega_1}$.
To make $X$ both HS and HL,
we shall apply the following lemma:

\begin{lemma}
\label{lemma-hshl}
Assume that $X$ is compact and for all closed $F \subseteq X$,
there is a compact metric $Y$ and a map $g: X \onto Y$
such that $g \rest g\iv(g(F)) \;:\; g\iv(g(F)) \;\onto\; g(F) $
is irreducible.  Then $X$ is both HS and HL.
\end{lemma}
\begin{proof}
By irreducibility, $g\iv(g(F)) = F$,
so that $F$ is a $G_\delta$ and $F$ is separable.
Thus, $X$ is a compact HL space in which all closed
subsets are separable, so $X$ is HS.
\end{proof}

In applying the lemma to $X = X_{\omega_1}$,
$g$ will be some $\pi^{\omega_1}_\alpha\res X$.
We shall use $\diamondsuit$ to capture all
closed  $F \subseteq Q^{\omega_1}$ so that
all closed $F \subseteq X$ will be considered.
This method was also employed in \cite{HK2},
which constructed some compacta which were HS and HL
but not Aronszajn.

As in standard inverse limit constructions, we inductively construct
$X_\alpha \subseteq Q^\alpha$, for $\alpha \le \omega_1$.
To ensure that $X$ will be Aronszajn, at each stage $\alpha< \omega_1$,
we carefully select a countable set
$\EE_\alpha \subseteq X_\alpha$ of ``expandable points'',
and at each stage $\beta > \alpha$, 
we construct $X_\beta\subseteq Q^\beta$ so that
$|(\sigma^{\beta}_\alpha)\iv\{x\}| = 1$
whenever $x \notin \EE_\alpha$.  Then the
$\LL_\alpha$ of Definition \ref{def-cube}
will be subsets of $\EE_\alpha$ and hence countable.

These preliminaries are included in
the following conditions:

\begin{conditions}
\label{cond-first-batch}
$X_\alpha$,  for $\alpha  \le \omega_1$, and
$\PP_\alpha, F_\alpha, \EE_\alpha, q_\alpha$,
for $0 < \alpha  < \omega_1$, satisfy:
\begin{enumctr}
\itemctr 
 Each $X_\alpha$ is a closed subset of $Q^\alpha$.
\itemctr 
 $\pi^\beta_\alpha(X_\beta) = X_\alpha$ whenever
$\alpha \le \beta \le \omega_1$.
\itemctr
 $\PP_\alpha$ is a countable family of closed subsets 
of $X_\alpha$, and $F_\alpha \in \PP_\alpha$.
\itemctr
For all $P \in \PP_\alpha$: \\
a.
$\sigma^{\alpha+1}_\alpha \res ((\sigma^{\alpha+1}_\alpha) \iv (P)) \;:\;
(\sigma^{\alpha+1}_\alpha)\iv(P) \;\onto\; P$ is irreducible, and \\
b.
$(\sigma^{\beta}_\alpha)\iv(P) \in \PP_{\beta}$ whenever
$\alpha \le \beta < \omega_1$. 
\itemctr 
For all closed $F \subseteq X$, there is an
$\alpha$ with $0 < \alpha < \omega_1$ such that
$\sigma^{\omega_1}_\alpha(F) = F_\alpha$.
\itemctr 
$\EE_\alpha$ is a countable dense subset of $X_\alpha$,
and $q_\alpha \in  \EE_\alpha$.
\itemctr 
$\EE_\beta \subseteq (\sigma^{\beta}_\alpha)\iv(\EE_\alpha)$
whenever $0 < \alpha \le \beta < \omega_1$.
\itemctr 
$|(\sigma^{\alpha + 1}_\alpha)\iv\{x\}| = 1$
whenever $0 < \alpha < \omega_1$ and
$x \in X_\alpha \backslash \{q_\alpha\}$.
\itemctr 
$|(\sigma^{\alpha + 1}_\alpha)\iv\{q_\alpha\}| > 1$.
\end{enumctr}
\end{conditions}

We discuss below how to satisfy these conditions.
Conditions (1) and (2) simply determine our
$X = X_{\omega_1}  \subseteq Q^{\omega_1}$ with
each $X_\alpha=\pi^{\omega_1}_\alpha(X)$.
$\diamondsuit$ is used for (5).
Constructing an $X$ that satisfies Conditions (1 - 9)   
is enough to prove
Theorem \ref{thm-Aronsz-basic}:

\begin{lemma}
Conditions $(1 - 9)$ imply that $X = X_{\omega_1}$ is an
Aronszajn compactum and is both HS and HL.
\end{lemma}
\begin{proof}
By (4) and induction on $\beta$,
$\sigma^{\beta}_\alpha \res ((\sigma^{\beta}_\alpha) \iv (P)) :
(\sigma^{\beta}_\alpha)\iv(P) \onto P$ is irreducible
whenever $\alpha \le \beta \le \omega_1$ and $P \in \PP_\alpha$.
Then $X$ is HS and HL by Lemma \ref{lemma-hshl} and (5)(3).

By (6)(7)(8) and induction, 
$|(\sigma^{\beta}_\alpha)\iv\{x\}| = 1$
whenever $0 < \alpha \le \beta \le  \omega_1$
and $x \in X_\alpha \backslash \EE_\alpha$.
So, $\LL_\alpha :=
\{x \in X_\alpha : |(\sigma^{\omega_1}_\alpha)\iv\{x\}| > 1\} 
\subseteq \EE_\alpha$, which is countable by (6).

Finally, $w(X) = \aleph_1$ by (9), and
$\cchi(X) = \aleph_0$ because $X$ is HL.
\end{proof}

To obtain Conditions $(1 - 9)$, we must add some further
conditions so that the natural construction avoids contradictions.
For example, satisfying Conditions (6) and (7) at stage $\beta$ 
requires $\bigcap_{\alpha < \beta}
(\sigma^{\beta}_\alpha)\iv(\EE_\alpha) \ne \emptyset$.
So we add Conditions (10 - 12) below making
the $\EE_\alpha$ into the levels of a tree;
the selection of the $\EE_\alpha$  will resemble the
standard inductive construction of an Aronszajn tree.

The sets $F_\alpha$ may be scattered or even singletons.
This cannot be avoided, 
because we are using the $F_\alpha$ to ensure
that \emph{all} closed sets are $G_\delta$ sets, so that
$X$ is HL;  making just the perfect sets 
$G_\delta$ could produce a Fedorchuk space (as in \cite{HK3}),
which is not even first countable.
If $x \in P \in \PP_\alpha$ and $x$
is isolated in $P$,  then the irreducibility condition
in (4) requires that $|(\sigma^{\alpha + 1}_\alpha)\iv\{x\}| = 1$,
but that contradicts (9) if $x = q_\alpha$.
Now, if every point of $\EE_\alpha$
is isolated in some $P \in \PP_\alpha$, then
we cannot choose $q_\alpha \in \EE_\alpha$, as required by (6).
We shall avoid these problems by requiring that if
$x \in \EE_\alpha$ and $P \in \PP_\alpha$, then
either $x \notin P$ or $x$ is in the perfect kernel of $P$.
This can be ensured by choosing $F_\alpha$ first 
(as given by $\diamondsuit$), and then choosing $\EE_\alpha$;
for limit $\alpha$, our Aronszajn tree construction will give
us plenty of options for choosing the points of $\EE_\alpha$,
and we shall make $F_\alpha$ trivial for successor $\alpha$.
The additional conditions that handle this will employ
the notation in the following:

\begin{definition}
If $F$ is compact and not scattered, let $\ker(F)$ denote
the perfect kernel of $F$; otherwise, $\ker(F) = \emptyset$.
\end{definition}

To satisfy Condition (8), we construct
$X_{\alpha + 1}$ from $X_\alpha$
by choosing an appropriate 
$h_\alpha \in C(X_\alpha \backslash \{q_\alpha\}, Q)$,
and letting $X_{\alpha+1}=\cl(h_\alpha)$. 
Identifying 
$Q^{\alpha+1}$ with $Q^\alpha \times Q$ and
$h_\alpha$ with its graph,
$h_\alpha(x)$ is the $y \in Q$ such that
$x \cat y \in X_{\alpha+1}$.
Note that $h_\alpha$ is indeed continuous because its graph
is closed.  

Thus, to construct $X$ so that Conditions (1 - 9) are met,
we add the following:
\begin{conditions}
$h_\alpha$  and $r^n_\alpha$,
for $0 < \alpha  < \omega_1$ and $n < \omega$, satisfy:
\begin{enumctr}
\itemctr 
$(\sigma^{\beta}_\alpha)(\EE_\beta) = \EE_\alpha$
whenever $0 < \alpha \le \beta < \omega_1$.
\itemctr 
$|\EE_ {\alpha + 1}\cap(\sigma^{\alpha + 1}_\alpha)\iv\{q_\alpha\}| > 1$.
\itemctr 
If $x \in \EE_\alpha$,
then $(\sigma^{\alpha + n}_\alpha)(q_{\alpha + n}) = x$ for
some $n \in \omega$.
\itemctr 
$X_\alpha$ has no isolated points whenever $\alpha > 0$.
\itemctr 
$F_\alpha = \emptyset$ whenever $\alpha$ is a successor ordinal.
\itemctr 
$\PP_\beta = \{F_\beta\} \cup \{
(\sigma^{\beta}_\alpha)\iv(P): 0 < \alpha < \beta
\ \&\  P \in \PP_\alpha \}$.
\itemctr 
$\EE_\alpha \cap (P \backslash \ker(P)) = \emptyset$
whenever $P \in \PP_\alpha$.
\itemctr 
$r^n_\alpha \in X_\alpha \backslash \{q_\alpha \}$ and the sequence
$\langle r^n_\alpha : n \in \omega \rangle$ converges to $q_\alpha$.
\itemctr 
$h_\alpha \in C(X_\alpha \backslash \{q_\alpha\}, Q)$, and
$ X_{\alpha+1} = \cl(h_\alpha)$.
\itemctr 
If $q_\alpha \in P \in \PP_\alpha$,
then $ r^n_\alpha  \in \ker(P)$ for infinitely many $n$, and
every $y\in Q$ with $q_\alpha \cat y \in X_{\alpha + 1}$
is a limit point of the sequence
$\langle h_\alpha(r^n_\alpha ) :
n \in \omega \ \&\ r^n_\alpha  \in \ker(P) \rangle$.
\end{enumctr}
\end{conditions}

Observe that (10)(11)(12) will give us the following:

\begin{lemma}
\label{lem-EisL}
$\LL_\alpha = \EE_\alpha$ whenever $0 < \alpha < \omega_1$ .
\end{lemma}

In the tree $T(X)$, although only
the node $q_\alpha \in \LL_\alpha$  has more than one successor
in $\LL_{\alpha + 1}$, (12) ensures that 
at limit levels $\gamma$, there are $2^{\aleph_0}$
choices for the elements of $\EE_\gamma$, so that we may
avoid the points in 
$F_\gamma \backslash \ker(F_\gamma)$, as required by (16).

By (14)(15), $\emptyset \in \PP_\alpha$ for all $\alpha > 0$,
and non-empty sets are added into the $\PP_\alpha$ only 
at limit $\alpha$.

The following proof gives the bare-bones construction; refinements 
of it produce the spaces of Theorem \ref{thm-Aronsz-refined}.

\begin{proofof}{Theorem \ref{thm-Aronsz-basic}}
Before we start, use $\diamondsuit$ to choose a closed
$\widetilde F_\alpha \subseteq Q^\alpha$ for each $\alpha < \omega_1$,
so that
$\{\alpha < \omega_1 : \pi^{\omega_1}_\alpha(F)=\widetilde F_\alpha\}$ 
is stationary for all closed $F \subseteq Q^{\omega_1}$.
To begin the induction:
$X_0$ must be $\{\emptyset\} = Q^0$, and 
$\PP_\alpha, F_\alpha, \ldots\ldots$ are only defined
for $\alpha > 0$.

Now, fix $\beta$ with $0 < \beta < \omega_1$, and assume
that all conditions have been met below $\beta$.
We define in order
$X_\beta$, 
$F_\beta$,
$\PP_\beta$,
$\EE_\beta$,
$q_\beta$,
$r^n_\beta$,
$h_\beta$.

If $\beta$ is a limit, then $X_\beta$ is determined
by (1)(2) and the $X_\alpha$ for $\alpha < \beta$.
$X_1$ can be any perfect subset of $Q^1$.
If $\beta = \alpha+1 \ge 2$, then
$X_\beta = \cl(h_\alpha)$, as required by (18).
Now let $F_\beta = \widetilde F_\beta$ if
$\widetilde F_\beta \subseteq X_\beta$ and $\beta$ is a limit;
otherwise, let $F_\beta = \emptyset$.
$\PP_\beta$ is now determined by (15).

$\EE_1$ can be any countable dense subset of $X_1$.
If $\beta = \alpha+1 \ge 2$, 
let $\EE_\beta =
(\sigma^{\beta}_\alpha) \iv (\EE_\alpha \backslash \{q_\alpha\} )
\cup \DD_\beta$, where $\DD_\beta$ is any subset of
$(\sigma^{\beta}_\alpha)\iv \{q_\alpha\}$ such that
$2 \le |\DD_\beta| \le \aleph_0$.
Observe that $\EE_\beta$ is dense in $X_\beta$
(without using $\DD_\beta$), so (6) is preserved,
and $\DD_\beta$ guarantees that (11) is preserved.
To verify (16) at $\beta$, note that by (15) at $\alpha$, every
non-empty set in $\PP_\beta$ is of the form
$\widehat P :=
(\sigma^{\beta}_\alpha) \iv (P)$ for some $P \in \PP_\alpha$.
So, if (16) fails at $\beta$, fix
$P \in \PP_\alpha$ and 
$x \in \EE_\beta \cap (\widehat P \backslash \ker(\widehat P))$.
Then $x \in (\sigma^{\beta}_\alpha)\iv \{q_\alpha\}$,
so $q_\alpha \in P$, and hence $q_\alpha \in \ker(P)$;
but then by (19), $x$ is a limit of a sequence of elements
of $\ker(\widehat P)$, so that $x \in \ker(\widehat P)$.

For limit $\beta$, let
$\EE_\beta = \{x^* : x \in \bigcup_{\alpha < \beta} \EE_\alpha\}$,
where, $x^*$, for $x \in \EE_\alpha$,  is
some $y \in X_\beta$ such that $\pi^\beta_\alpha(y) = x$
and $\pi^\beta_\xi(y) \in \EE_\xi$ for all $\xi < \beta$.
Any such choice of the $x^*$ will satisfy (10).
But in fact, using (11)(12), for each such $x$ there
are $2^{\aleph_0}$ possible choices of $x^*$,
so we can satisfy (16) by avoiding the countable
sets $P \backslash \ker(P)$ for $P \in \PP_\beta$.

To facilitate (12), list each $\EE_\alpha$ as
$\{e_\alpha^j : j \in \omega\}$; let
$e_0^j = \emptyset \in X_0$.
Then, if $\beta$ is a successor ordinal of the form
$\gamma + 2 ^i 3^j$, where $\gamma$ is a limit or $0$,
choose $q_\beta \in \EE_\beta$ so that
$\sigma^\beta_{\gamma + i}(q_\beta) = e_{\gamma + i}^j$.
For other $\beta$, $q_\beta \in \EE_\beta$ can
be chosen arbitrarily.

Next, we may choose the $r^n_\beta$ to satisfy (19)
because if $q_\beta \in P \in \PP_\beta$,
then $q_\beta \in \ker(P)$ by (16), so that 
$q_\beta$ is also a limit of points in $\ker(P)$.

Finally, we must choose
$h_\beta \in  C(X_\beta \backslash \{q_\beta\}, Q)$.
Conditions (18)(19) only require that 
$h_\beta$ have a discontinuity at $q_\beta$ with the property that
every limit point of the function at $q_\beta$ is also
a limit of each of the sequences 
$\langle h_\beta(r^n_\beta ) :
n \in \omega \ \&\ r^n_\beta  \in \ker(P) \rangle$.
Since $X_\beta$ is a compact metric space
with no isolated points, we may accomplish this by
making every point of $Q$ a limit point of each
$\langle h_\beta(r^n_\beta ) :
n \in \omega \ \&\ r^n_\beta  \in \ker(P) \rangle$.
\end{proofof}

If we choose each $h_\beta$ as above and also set $X_1 = Q$,
then our $X$ will be connected, and it is fairly easy to choose
the $h_\beta$ so that $X$ fails to be locally connected.
The next theorem shows how
to make $X$ connected and locally connected.
We construct $X$ so that each $X_\alpha$ 
is homeomorphic to the \textit{Menger sponge}, $\MS$,
and all the maps $\sigma^\beta_\alpha$ are monotone.
The Menger sponge \cite{Me} is a one dimensional
locally connected metric continuum; the properties of $\MS$ used
in inductive constructions such as these are summarized in \cite{HK3},
which contains further references to the literature.
A map is \emph{monotone} iff all point inverses are connected.
Monotonicity of the $\sigma^\beta_\alpha$ will imply that 
$X$ is locally connected.

At successor stages, 
to construct $X_{\alpha+1} \cong \MS$,   
we assume that $X_\alpha \cong \MS$ and 
apply the following special case of 
Lemmas 2.7 and 2.8 of \cite{HK3}:  

\begin{lemma}
\label{lem-MS}
Assume that $q \in X \cong \MS$ and that for each $j \in \omega$,
the sequence
$\langle r^n_j : n \in \omega \rangle$ converges to $q$,
with each $r^n_j \ne q$.  Let $\pi : X \times [0,1] \onto X$
be the natural projection.  Then there is a $Y \subseteq  X \times [0,1]$
such that:
\begin{itemizz}
\item[1.] $Y \cong \MS$ and $\pi(Y) = X$.
\item[2.] $|Y \cap \pi\iv\{x\}| = 1$ for all $x \ne q$.
\item[3.] $\pi\iv\{q\} = \{q\}\times [0,1]$.
\item[4.] Let $Y \cap \pi\iv\{r^n_j\} = \{(r^n_j, u^n_j)\}$.  Then,
for each $j$, every point in $[0,1]$ is a limit point
of $\langle u^n_j : n \in \omega \rangle$.
\end{itemizz}
\end{lemma}

Constructing $X$ as such an inverse limit of 
Menger sponges will make $X$ one dimensional.
The results quoted
from \cite{HK3} about $\MS$ were patterned on an earlier
construction of van Mill \cite{vM}, which involved an inverse
limit of Hilbert cubes;   
replacing $\MS$ by $Q$  here
would yield an infinite dimensional version
of this Aronszajn compactum.
The following summarizes several possibilities for 
$X$ and its associated tree:
\begin{theorem}
\label{thm-Aronsz-refined}
Assume $\diamondsuit$.
For each of the following $2 \cdot 3 = 6$ possibilities,
there is an Aronszajn compactum $X$
with associated Aronszajn tree $T$ such that $X$ is HS and HL.
Possibilities for $T$:
\begin{itemizz}
\item[a.] $T$ is Suslin.
\item[b.] $T$ is special.
\end{itemizz}
Possibilities for $X$:
\begin{itemizz}
\item[$\alpha$.] $\dim(X) = 0$.
\item[$\beta$.] $\dim(X) = 1$ and $X$ is connected and locally connected.
\item[$\gamma$.] $\dim(X) = \infty$ and $X$ is connected and locally connected.
\end{itemizz}
\end{theorem}
\begin{proof}
We refine the proof of Theorem \ref{thm-Aronsz-basic},
To obtain $(a)$ or $(b)$, the refinement is in the choice
of the $\EE_\beta$ for limit $\beta$.
To obtain $(\alpha)$ or $(\beta)$ or $(\gamma)$,
the refinement is in the choice of $X_1$ and the functions $h_\alpha$.
Since these refinements are independent of each other,
the discussion of $(a)(b)$ is unrelated to the discussion of
$(\alpha)(\beta)(\gamma)$.

For $(a)$: We use $\diamondsuit$ to kill all 
potential uncountable maximal antichains $A \subset T$.
Fix a sequence $\langle A_\alpha : \alpha < \omega_1\rangle$
such that each $A_\alpha$ is a countable subset of $Q^{< \alpha}$ and
such that for all $A \subseteq Q^{< \omega_1}$:
if each $A \cap Q^{< \alpha}$ is countable, then
$\{\alpha < \omega_1 : A\cap Q^{< \alpha} = A_\alpha\}$ is stationary.

Let $T_\beta= \bigcup\{\LL_\alpha: \alpha< \beta\}
= \bigcup\{\EE_\alpha: \alpha< \beta\}$ (see Lemma \ref{lem-EisL}),
and use $\tleft$ for the tree order.
For each limit $\beta < \omega_1$,
modify the construction of $\EE_\beta$
in the proof of Theorem \ref{thm-Aronsz-basic} as follows:
We still have
$\EE_\beta = \{x^* : x \in T_\beta\}$,
where, $x^*$, for $x \in T_\beta$, is chosen so that $x \tleft x^*$ and
$x^*$ defines a path through $T_\beta$.
But now, \emph{if} $A_\beta \subseteq T_\beta$ \emph{and}
$A_\beta$ is  a  maximal  antichain  in $T_\beta$, 
\emph{then} make sure that each $x^*$ is above some element of $A_\beta$.
To do this, use maximality of $A_\beta$ first to choose
$x^\dag \in T_\beta$ so that $x \tleft x^\dag$ and $x^\dag$
is above some element of $A_\beta$, and then choose $x^*$ so that
 $x \tleft x^\dag \tleft x^*$.
There are still $2^{\aleph_0}$ possible choices for $x^*$,
so we can satisfy (16) by avoiding the countable
sets $P\! \setminus\!\ker(P)$ as before.
Now, the usual argument shows that $T$ is Suslin.

For $(b)$:  Let $\Lim$ denote the set of countable limit ordinals,
and let $T^\Lim = \bigcup\{\LL_\alpha: \alpha  \in \Lim\} =
\bigcup\{\EE_\alpha: \alpha  \in \Lim\}$.
To make $T$ special,
inductively define an order preserving map 
$\varphi : T^\Lim \to \QQQ$.
To make the induction work, we also assume inductively:
\[
\forall \gamma,\beta \in \Lim \,  \forall x \in \LL_\gamma\,
\forall q \in \QQQ\;
[\gamma < \beta \; \&\;  q > \varphi(x)  
\ \to \ \exists y \in \LL_\beta \, [x \tleft y \; \&\; \varphi(y) = q]]
\tag{$*$}
\]
To start the induction, $\varphi \res \LL_\omega : \LL_\omega \to \QQQ$
can be arbitrary.

For $\beta = \alpha + \omega$, where $\alpha$ is a limit ordinal:
First, determine the $x^*$ exactly as in the proof of
Theorem \ref{thm-Aronsz-basic}.  Then, note that for each $x \in \LL_\alpha$,
the set $S_x := \{y \in \EE_\beta : x \tleft y\}$ has size $\aleph_0$,
so we can let $\varphi \res S_x$ map $S_x$ \emph{onto}
$\QQQ \cap (\varphi(x), \infty)$.

For $\beta < \omega_1$ which is a limit of limit ordinals:
Let 
\[
\EE_\beta = \{x^*_q : x \in T_\beta \;\&\;
q  \in \QQQ \cap (\varphi(x), \infty)\}
\ \ ,
\]
where each $x^*_q$  is chosen so that
$x \tleft x^*_q$ and
$x^*_q$ defines a path through $T_\beta$ and the $x^*_q$ are all different
as $q$ varies.
We let $\varphi(x_q^*) = q$, which will clearly preserve $(*)$,
but we must make sure that $\varphi$ remains order preserving.
For this, choose $x_q^*$ so that
$q > \sup\{\varphi(z) : z \in T^\Lim \; \&\ z \tleft x_q^*\}$.
Such a choice is possible using $(*)$ on $T_\beta$.
As before, there are $2^{\aleph_0}$ possible choices of $x_q^*$,
so we can still avoid the countable sets $P\!\setminus\!\ker(P)$.

For $(\alpha)$, just make sure that $X_\alpha$
is homeomorphic to the Cantor set $2^\omega$
whenever $0 < \alpha < \omega_1$.
In view of (13), this is equivalent to making $X_\alpha$
zero dimensional.  For $\alpha = 1$, we simply choose
$X_1$ so that $X_1 \cong 2^\omega$.
Then, for larger $\alpha$, just make sure that in (9),
we always have
$|(\sigma^{\alpha + 1}_\alpha)\iv\{q_\alpha\}| = 1$,
which will hold if in (18), we choose 
$h_\alpha \in C(X_\alpha \backslash \{q_\alpha\}, 2)$
(identifying $2 = \{0,1\}$ as a subset of $Q$).
To make this choice, and satisfy (19):
First, let $A_j$, for $j \in \omega$, be disjoint
infinite subsets of $\omega$ such that for each $P \in \PP_\alpha$,
if $q_\alpha \in P$  then for some $j$,
$ r^n_\alpha  \in \ker(P)$ for all $n \in A_j$.
Next, let $X_\alpha = K_0 \supset K_1 \supset K_2 \supset \cdots$,
where each $K_i$ is clopen, $\bigcap_i K_i = \{q_\alpha\}$,
and, for each $j$, there are 
infinitely many even $i$ and
infinitely many odd $i$ such that
$K_{i} \backslash K_{i + 1} \cap \{  r^n_\alpha : n \in A_j\} \ne \emptyset$.
Now, let $h_\alpha$ be
$0$ on $K_{i} \backslash K_{i + 1}$ when $i$ is even and 
$1$ on $K_{i} \backslash K_{i + 1}$ when $i$ is odd.

For $(\beta)$, 
construct $X$ so that each $X_\alpha$ 
is homeomorphic to the \textit{Menger sponge}, $\MS$,
and all the maps $\sigma^\beta_\alpha$ are monotone.
Then 
$\dim(X) = 1$ will follow from the fact that $X$ is an inverse
limit of one dimensional spaces.

For monotonicity of the $\sigma^\beta_\alpha$,
it suffices to ensure that each
$\sigma^{\alpha + 1}_\alpha$ is monotone.  By Condition (8),
that will follow if we make
$(\sigma^{\alpha + 1}_\alpha)\iv\{q_\alpha\}$ connected;
in fact we shall make
$(\sigma^{\alpha + 1}_\alpha)\iv\{q_\alpha\}$ homeomorphic to $[0,1]$,
as in the proof of Theorem \ref{thm-Aronsz-basic}.
But we also need to verify inductively that 
$X_\alpha \cong \MS$.  At limits, this follows from
Lemma 2.5 of \cite{HK3}.  At successor stages, 
we assume that $X_\alpha \cong \MS$ and
identify $[0,1]$ as a subspace of $Q$, so that
$X_{\alpha+1}$ may be the $Y$ of Lemma \ref{lem-MS}.

$(\gamma)$ is proved analogously
to $(\beta)$.  
Construct $X_\alpha \cong Q$ rather than $\MS$,
applying the results about $Q$ in \cite{vM}\S3.
As in \cite{vM}\S2, all the 
$\sigma^\beta_\alpha$ are cell-like $Z^*$-maps.
\end{proof}

\section{Chains of Clopen Sets}

The double arrow space has an uncountable chain
(under $\subset$) of clopen sets of real type.
This cannot happen in an Aronszajn compactum:

\begin{lemma}
\label{lemma-no-real-type}
If $X$ is an Aronszajn compactum and $\EE$ is an
uncountable chain 
of clopen subsets of $X$, then $\EE$ cannot be of real type.
\end{lemma}
\begin{proof}
Suppose that $\EE$ is such a chain.  Deleting some elements
of $\EE$, we may assume that $(\EE, \subset)$
is a dense total order.
Let $\DD$ be a countable dense subset of $\EE$.
Since $X$ is an Aronszajn compactum, there is a map
$\varphi : X \onto Z$, where $Z$ is a compact metric space,
$A = \varphi\iv(\varphi(A))$ for all $A \in \DD$,
and $\{y \in Z : |\varphi\iv\{y\}| > 1\}$ is countable.
Since $\DD$ is dense in $\EE$, the sets
$\varphi(B)$ for $B \in \EE$ are all different.
Each $\varphi(B)$ is closed, and only countably many
of the $\varphi(B)$ can be clopen.
Whenever $\varphi(B)$ is not clopen, choose 
$y_B \in \varphi(B) \cap \varphi(X \backslash B)$.
Since $\DD$ is dense in $\EE$,
these $y_B$ are all different points, so
there are uncountably many such $y_B$.
But $\varphi\iv\{y_B\}$ meets both $B$ and $X \backslash B$,
so each $|\varphi\iv\{y_B\}| \ge 2$, a contradiction.
\end{proof}

Note that if this argument is applied with a chain of clopen sets in
the double arrow space, then
the  $|\varphi\iv\{y_B\}|$ will be exactly $2$.

\begin{lemma}
If $X$ is any separable space, 
and $\EE$ is an uncountable chain 
of clopen subsets of $X$, then $\EE$ must be of real type.
\end{lemma}
\begin{proof}
If $D \subseteq X$ is dense, then 
$(\EE, \subset)$ is isomorphic to a chain in $(\PP(D), \subset)$.
\end{proof}

\begin{corollary}
If $X$ is a separable Aronszajn compactum and $\EE$ is a
chain of clopen subsets of $X$, then $\EE$ is countable.
\end{corollary}

Note that if $X$ is a zero dimensional compacted Aronszajn
line which is also Suslin (see Lemma \ref{lemma-sus-ar}),
then $X$ has
uncountable chain of clopen sets, but $X$ is not separable.

\end{document}